%% file: bqsg2025_workshop.tex
\documentclass[11pt]{article}
\usepackage{bqsg2025_workshop}
\usepackage{amsmath, amssymb, amsthm}
\usepackage[T1]{fontenc}
\usepackage{mathtools}
\usepackage{xcolor}
\definecolor{darkblue}{rgb}{0.0, 0.0, 0.8}
\usepackage[colorlinks=true, allcolors=darkblue]{hyperref}

\renewcommand{\thesection}{}
\renewcommand{\thesubsection}{\arabic{subsection}}

\newtheorem{theorem}{Theorem}
\newtheorem{proposition}[theorem]{Proposition}
\newtheorem{question}[theorem]{Question}

\newtheorem{lemma}{Lemma}
\newtheorem{definition}[lemma]{Definition}

\theoremstyle{definition}
\newtheorem{remark}[lemma]{Remark}
\def\d{\mathrm{d}}
\usepackage{url} %

\title{Open problems in billiards and quantitative symplectic geometry}

\author{
\textbf{Bernhard Albach},
\textbf{Jean-Fran\c{c}ois Barraud},
\textbf{Misha Bialy},\\
\textbf{Johanna Bimmermann},
\textbf{Mihai Damian},
\textbf{Umberto Hryniewicz},\\
\textbf{Vincent Humilière},
\textbf{Boris Khesin}, 
\textbf{Agustin Moreno}, 
\textbf{Alexandru Oancea}\\
\textbf{Olga Paris-Romaskevich},
\textbf{Alfonso Sorrentino},  
\textbf{Serge Tabachnikov}
}
  
\date{}

\begin{document}

\maketitle

\footnotetext{Extended author information available on the last page of the article}

\begin{abstract}
This document collects contributions to the Open Problem List in Billiards and Quantitative Symplectic Geometry, compiled following discussions during the workshop ``Billiards and quantitative symplectic geometry'' that took place at the University of Heidelberg on July 14--18, 2025.
\end{abstract}

\textbf{Acknowledgments:}
As organizers of the above-mentioned workshop and editors of this compilation, Ana Chávez Cáliz, Lina Deschamps, and Levin Maier thank all participants for their engagement and the lively discussions and thank, in particular, the contributors of the open problems presented here.\\
This document uses the official style file for submissions to the Workshop on Geometry and Topology in Machine Learning 2025. Thanks are due to Diaaeldin Taha for writing and generously sharing the code used to generate it.\\
A.C.C., L.D. and L.M. acknowledge funding from the Deutsche Forschungsgemeinschaft (DFG, German Research Foundation) – 281869850 (RTG 2229), 390900948 (EXC-2181/1), and 281071066 (TRR 191).

\input{Contributions/Bialy} 
\input{Contributions/Bimmermann} 
\input{Contributions/Umberto} 
\input{Contributions/Khesin} 
\input{Contributions/Moreno} 
\input{Contributions/Oancea} 
\input{Contributions/Paris-Romaskevich} 
\input{Contributions/Sorrentino}
\input{Contributions/Tabachnikov} 

\bibliographystyle{abbrv}
\bibliography{bqsg2025_workshop}

\section{Authors and Affiliations}

Bernhard Albach\\
\texttt{albach@mathga.rwth-aachen.de},\\
Lehrstuhl f\"ur Geometrie und Analysis, RWTH Aachen, Germany. \bigskip 

Jean-Fran\c{c}ois Barraud\\
\texttt{barraud@math.univ-toulouse.fr},\\
Institut de Mathématiques de Toulouse (IMT), Université Paul Sabatier, Toulouse, France. 
\bigskip 

Misha Bialy\\
\texttt{bialy@tauex.tau.ac.il} \\
School of Mathematical Sciences, Tel Aviv University, Israel \bigskip

Johanna Bimmermann\\
\texttt{johanna.bimmermann@maths.ox.ac.uk} \\
Mathematical Institute, University of Oxford, U.K. 
\bigskip 

Mihai Damian\\
\texttt{mihai.damian@math.unistra.fr} \\
Institut de Recherche Mathématique Avancée (IRMA), Université de Strasbourg, France. \bigskip

Umberto Hryniewicz\\
\texttt{hryniewicz@mathga.rwth-aachen.de} \\
Lehrstuhl f\"ur Geometrie und Analysis, RWTH Aachen, Germany. \bigskip

Vincent Humilière\\
\texttt{vincent.humiliere@imj-prg.fr} \\
Institut de Mathématiques de Jussieu-Paris Rive Gauche (IMJ-PRG), Sorbonne Université, Paris, France. 
\bigskip

Boris Khesin\\
\texttt{khesin@math.toronto.edu}\\
University of Toronto, Canada.
\bigskip 

Agustin Moreno\\
\texttt{agustin.moreno2191@gmail.com} \\
University of Heidelberg, Germany.
\bigskip

Alexandru Oancea\\
\texttt{oancea@unistra.fr} \\
Institut de Recherche Mathématique Avancée (IRMA), Université de Strasbourg, France.
\bigskip 

Olga Paris-Romaskevich\\
\texttt{paro@math.univ-lyon1.fr} \\
CNRS, ICJ UMR5208, École Centrale de Lyon, INSA Lyon, Université Claude Bernard \\
Lyon 1, Université Jean Monnet, $69622$ Villeurbanne, France.
\bigskip 

Alfonso Sorrentino\\
\texttt{sorrentino@mat.uniroma2.it} \\
Department of Mathematics, University of Rome Tor Vergata, Italy.
\bigskip

Serge Tabachnikov\\
\texttt{tabachni@math.psu.edu} \\
Pennsylvania State University, USA.

\section{Editors}

Ana Chávez Cáliz\\
\texttt{ana.chavez@im.unam.mx} \\
IMATE Unidad Cuernavaca, Mexico.
\bigskip

Lina Deschamps\\
\texttt{ldeschamps@mathi.uni-heidelberg.de} \\
Universität Heidelberg, Germany.
\bigskip

Levin Maier\\
\texttt{lmaier@mathi.uni-heidelberg.de} \\
Universität Heidelberg, Germany.
 
\end{document}

%% file: Contributions/Bialy.tex
\section*{Misha Bialy}



The notion of total integrability for systems with toric configuration space means the existence of a foliation of the phase space by invariant tori, which are graphs. 
If a region of the phase space is foliated by invariant graphs, we say that the system is totally integrable on this region. The orbits lying on these graphs are action-minimizing. Previous rigidity results \cite{bialy-mironov-survey}
show that total integrability is a very rigid notion.

Here, we want to consider two billiard models and two continuous-time systems.
In each case, the question is whether there are new examples of totally integrable systems.

\subsection*{Magnetic billiards for strong magnetic fields.}

Magnetic billiards were introduced and studied in \cite{berry-robnik}. In this model, the billiard ball moves under the influence of a constant magnetic field along Larmor circles (for a magnetic field of magnitude $B$, the Larmor circles have radius $1/B$). The only known example of an integrable magnetic billiard is a circular billiard. An algebraic approach to this problem was discussed in \cite{magnetic1}, \cite{magnetic2}. Numerical experiments were performed in \cite{albers}. Moreover, the approach of total integrability works well for convex magnetic billiard in a {\it weak} constant magnetic field \cite{bialy-magnetic-hopf}. 

{\bf Question 1.} {\it Is it true that the only totally integrable magnetic billiard in a {\it strong} constant magnetic field is necessarily circular?}

\subsection*{Minkowski billiard.}

The model of Minkowski and, more generally, Finsler billiards, was introduced and studied by Gutkin and Tabachnikov in \cite{GUTKIN2002277}. It plays an important role in a recent study by Artstein-Avidan, Karasev, and Ostrover \cite{Artstein_Avidan_2014} on Mahler's conjecture. Consider a smooth convex curve $\gamma\subset \mathbb R^2$, and a Minkowski (not necessarily symmetric) norm $N$. The Minkowski billiard in $\gamma$ with respect to the norm $N$ is defined as a twist map that corresponds to the generating function \[L(s_0,s_1)=N(\gamma(s_1)-\gamma(s_0))\,,\] and the corresponding variational principle $	\mathcal L\{s_n\}=\sum_{n} L(s_n,s_{n+1}),\ $ where $\gamma$ is parametrized by $N$-arc length parameter (meaning, $N(\dot{\gamma}(s))=1$). Thus, the Birkhoff billiard is a special case of the Minkowski billiard when $N$ is chosen to be a Euclidean norm.

{\bf Question 2.} (Communicated by Yaron Ostrover) {\it Suppose that the Minkowski billiard, corresponding to the norm $N$, in a unit ball of $N$ is totally integrable. Does this imply that the norm $N$ is necessarily Euclidean?}

This question was answered affirmatively in \cite{IMRN} for billiard tables with a finite symmetry group. 
\medskip

\subsection*{Newton equations with $1\frac 12$ degrees of freedom, with periodic potential.}

Newton's equation $\ddot u=-V'_u(u,t)$ is the Euler-Lagrange equation associated with the Lagrangian $L(u,\dot u, t)=\frac 12\dot u^2-V(u,t)$. Assume that $V$ is periodic both in $u$ and in $t$. The corresponding Hamiltonian flow acts on $T^*\mathbb S^1\times\mathbb S^1$. One can prove that this flow is totally integrable on the whole phase space if and only if $V$ does not depend on $u$. 
Moreover, if $V=V(nu+mt)$, with $ m\neq0, n\in\mathbb Z$, then there exists an integral of motion that is quadratic in momentum $p$, and hence the system is totally integrable in a neighborhood of infinity, $|p|>c$, for sufficiently large $c$.

{\bf Question 3.} {\it Are there other potential functions $V$ such that the system is totally integrable at infinity?}
\medskip

\subsection*{An elliptic PDE}

Consider an elliptic PDE  $\Delta u=-V'_u(u, x_1, \ldots ,x_n)$, where $\Delta=\sum_{i=1}^{n}\partial^2_{x_i}$ and $V$ is a 1-periodic function of $u, x_1,\ldots ,x_n$. This equation is the Euler-Lagrange equation for the functional $\int \left( \frac 12|\nabla u|^2-V(u, x_1, \ldots, x_n)\right) dx_1 \ldots dx_n$. Aubry-Mather theory, as well as KAM theory for this equation, were developed by J. Moser \cite{moser} and V. Bangert \cite{bangert}. If $V$ does not depend on $u$, we have the equation $\Delta u=0$. In this case, for any rotation vector $\alpha\in \mathbb R^n$, there exists a family of minimal solutions of the form $u=\alpha\cdot x+c$. Thus, for every $\alpha$, the graphs of these solutions form a foliation of $\mathbb T^{n+1}(u,x_1, \ldots, x_n)$ by minimals (hyperplanes of slope $\alpha$).

{\bf Question 4.} {\it Let $V$ be a 1-periodic potential. Suppose that for every $\alpha \in \mathbb R^n$ there exists a foliation of $\mathbb T^{n+1}$ by graphs of minimal solutions of slope $\alpha$. Is it true that, in this case, $V$ does not depend on $u$?}

For compactly supported potentials $V$, a variational version of this question was approached in \cite{bialy-mackay}.

\bigskip

{\it Acknowledgments.} The research was supported by ISF grant 974/24.

%% file: Contributions/Bimmermann.tex


\section*{Johanna Bimmermann}

The \emph{Hofer--Zehnder capacity} was introduced by Hofer and Zehnder \cite{HZ90} and is a numerical symplectic invariant measuring the size of a symplectic manifold in terms of Hamiltonian dynamics. More precisely, it assigns a number to any symplectic
manifold $(M,\omega)$:
$$
c_{\mathrm{HZ}}(M,\omega)
:= \sup\lbrace \, \max H \,\mid \, H\in C^\infty_0(M;\mathbb{R}_{\ge 0})\ \text{and $X_H$ has no fast periodic orbits} \rbrace.
$$
Here a \emph{fast periodic orbit} is a solution $\gamma:\mathbb{R}/T\mathbb{Z} \longrightarrow M,\ \dot\gamma = X_H(\gamma)$, with $T<1$. In words the Hofer--Zehnder capacity determines how much a Hamiltonian on $(M,\omega)$ can oscillate before fast periodic orbits appear.

Finiteness of $c_{\mathrm{HZ}}(M,\omega)$ implies the Weinstein
conjecture\footnote{i.e.,\ existence of a periodic orbit of the characteristic foliation} on stable
closed hypersurfaces $S\subset M$ (see~\cite[Ch.\ 4, Thm.\ 5]{HZ95}). This indicates that finiteness is a strong statement and is usually difficult to prove. For example, it is unknown whether the Hofer--Zehnder capacity of compact
domains in cotangent bundles $(T^{*}N,d\lambda)$ is finite. Already for $N$ a closed surface of genus $g\ge 2$, this is open.

Fix a hyperbolic metric $g$ on $N$ and denote by $D^{*}_{R}N$ the
radius-$R$ disk cotangent bundle. Then any compact domain
$K\subset T^{*}N$ is contained in $D^{*}_{R}N$ for $R$ large enough, and
scaling the fibers yields a symplectomorphism
$$
(D^{*}_{R}N,d\lambda)\;\cong\;(D^{*}_{1}N,\,R\,d\lambda).
$$
Hence, it is enough to prove finiteness for the unit disk cotangent bundle
$D^{*}N:=D^{*}_{1}N$.

\paragraph{Question:}\begin{center}
    Let $(N,g)$ be an orientable closed hyperbolic surface.\\ Does $c_{\mathrm{HZ}}\bigl(D^{*}N,d\lambda\bigr)<\infty$ hold?
\end{center}

\vspace{1em}

We included the word ``orientable'' for the following reason. If the Euler
characteristic $\chi(N)$ is divisible by $4$ and $N$ is non-orientable,
then there exists a Lagrangian embedding $N\hookrightarrow (\mathbb{C}^{2},\omega_{0})$ \cite{Giv86}. By Weinstein's tubular neighborhood theorem this implies that, for some
$\varepsilon>0$, $(D^{*}_{\varepsilon}N,d\lambda)\hookrightarrow (\mathbb C^{2},\omega_{0})$ embeds symplectically, which implies finiteness.\\

We now present two arguments that might hint finiteness and two arguments that might hint infiniteness. These are not mathematically worked out and are based on discussions with Gabriele Benedetti, Dylan Cant, Stefan Nemirovski and Alexandru Oancea.

\medskip
\noindent\textbf{Finiteness.}

\begin{enumerate}
\item \emph{Capacity relative to the zero section is finite.}

There is a modification of the definition of $c_{\mathrm{HZ}}$ for which the answer to (Q) is yes. Require the Hamiltonians $H\in C_c^\infty(D^{*}N,\mathbb{R}_{\ge 0})$ to attain their maximum along the zero section $N\subset D^{*}N$, i.e.,\ $H|_{N}=const.=\max H$. It is a direct consequence of \cite{Wbr06} that
$$
c_{\mathrm{HZ}}(D^{*}N,N,d\lambda)=\operatorname{sys}<\infty,
$$
where $\operatorname{sys}$ denotes the length of the shortest closed (non-contractible) geodesic.

\item \emph{Displaceability of the zero section after removing a point.}

Displacement energy gives an upper bound for the Hofer--Zehnder capacity, and if the zero section would be displaceable then so is $D^{*}_{\varepsilon}N$ for $\varepsilon$ small. In particular, displaceability of the zero section yields finiteness of $c_{\mathrm{HZ}}(D^{*}N,d\lambda)$. Actually, the full zero-section $N\subset T^{*}N$ of a closed manifold $N$ is never displaceable, but it becomes displaceable if one removes an arbitrarily small disc $D_{\varepsilon}\subset N$. 

To see this, take a Morse function $f:N\to\mathbb{R}$ with all critical points in $D_{\varepsilon/2}$. Then there exists some $\delta>0$ such that $\vert \nabla f\vert \geq \delta$ on $N\setminus D_{\varepsilon}$. Hence, the Hamiltonian flow of $H(q,p)=\frac{1}{2}\vert p\vert^2+f(q)$ displaces $N\setminus D_{\varepsilon}\subset T^{*}N$. We obtain:
$$
c_{HZ}(D^*(N\setminus D_\varepsilon), d\lambda)<\infty,\quad \forall\varepsilon>0.
$$
\end{enumerate}

\medskip
\noindent\textbf{Infiniteness.}

\begin{enumerate}
\item \emph{Lack of pseudoholomorphic curves in $T^{*}N$.}

Most known bounds for the Hofer--Zehnder capacity ultimately come from
pseudoholomorphic curves, as first explained by Hofer--Viterbo \cite{HV92}.
Indeed one needs a form of uniruledness, morally, a
pseudoholomorphic curve through every point. For all closed surfaces $N$
for which finiteness of $c_{\mathrm{HZ}}(D^{*}N,d\lambda)$ is known,
uniruledness is ensured by the existence of a Lagrangian embedding
$N\hookrightarrow \Sigma$ into a ruled complex surface $\Sigma$; for example
$$
S^{2}\subset \mathbb{CP}^{1}\times\mathbb{CP}^{1},\qquad
\mathbb{RP}^{2}\subset \mathbb{CP}^{2},\qquad
T^{2}\subset \mathbb{CP}^{2},\ \ldots
$$
For orientable hyperbolic surfaces, however, no such embedding exists. Indeed, suppose that an orientable surface $N$ of genus $g\ge 2$ was embedded as a Lagrangian into a ruled complex surface $\Sigma$. Then the self-intersection satisfies
$$
[N]^{2} \;=\; -\chi(N) \;=\; 2g-2 \;>\; 0,
$$
since for a Lagrangian surface $L$ in a complex surface one has $[L]^{2}=-\chi(L)$. On the other hand, if $[\omega]$ denotes the Kähler class of $\Sigma$, the Lagrangian condition $\omega|_{N}=0$ implies
$$
[\omega]\cdot \mathrm{PD}[N] \;=\; 0.
$$
Thus $[\omega]$ and $\mathrm{PD}[N]$ are orthogonal classes of positive square in $H^{2}(\Sigma;\mathbb{R})$, which is impossible as the intersection form has exactly one positive eigenvalue \cite[Prop.\ III.18]{Bea96}. Consequently, an orientable hyperbolic surface cannot embed as a Lagrangian into a ruled complex surface.

\item \emph{No upper bounds from symplectic homology.}
The other existing strategies for proving finiteness of the Hofer--Zehnder capacity (or uniruledness) of cotangent bundles rely on symplectic homology. 
In particular, the existence of a \textit{dilation class} $\Delta x = 1$ \cite{BC25}, 
a nontrivial \textit{pair-of-pants product} $x \star y = 1$ \cite{Iri14}, 
or the vanishing of symplectic homology with (differential graded) local coefficients \cite{AFO17,BDHO-cotangent} implies such finiteness. Except for Irie's pair-of-pants product, all these criteria provide upper bounds for the $\pi_1$-sensitive Hofer--Zehnder capacity, which only requires all contractible orbits to be slow. 
It is easily seen that for a hyperbolic surface, $c_{HZ}^0(D^*N, d\lambda)=\infty$. For example, the geodesic flow, generated by the Hamiltonian $E(q,p)=\tfrac{1}{2}\lvert p\rvert_q^2$, has no contractible periodic orbits. Hence, none of these criteria can be used to prove finiteness for hyperbolic surfaces.\\

For cotangent bundles there is a ring isomorphism between symplectic homology and the homology of the free loop space \cite{Abo15}.
For a closed hyperbolic surface $N$, the latter can be computed explicitly since $N$ is a $K(\pi,1)$ \cite[Ch.\ 1.4]{LO15}:
$$
H_*(\mathcal{L}N,\mathbb{Z})
\cong
\bigoplus_{[g]\in \mathrm{Conj}} H_*(K(C_\pi(g),1),\mathbb{Z}),
$$
where $C_\pi(g)$ denotes the centralizer of $g\in\pi=\pi_1(N)$.  
Here $C_\pi(e)=\pi$ and $C_\pi(g)\cong\mathbb{Z}$ for $g\neq e$, giving
$$
H_*(\mathcal{L}N,\mathbb{Z})
\cong
H_*(N,\mathbb{Z})
\oplus
\bigoplus_{[g]\neq[e]} H_*(S^1,\mathbb{Z}).
$$
This simple structure immediately rules out the existence 
of a nontrivial loop product $x\star y = 1$, as $\star: H_i(\mathcal{L}N,\mathbb{Z})\times H_j(\mathcal{L}N,\mathbb{Z})\to H_{i+j-2}(\mathcal{L}N,d\lambda)$, and the unit is the only degree $2$ element.
\end{enumerate}

We conclude with a possible strategy to prove infiniteness by linking
quantitative symplectic geometry with billiards.

\textbf{Lower bounds via hyperbolic billiards.} If $\Omega\subset N$ is open and simply connected, it lifts to an open
subset of the hyperbolic plane $\mathbb{H}^{2}$ with compact closure, let
$\Omega\text{-}\mathrm{sys}$ denote the length of the shortest billiard trajectory in
$\Omega$. Then
$$
\Omega\text{-}\mathrm{sys}\le c_{\mathrm{HZ}}(D^{*}N,d\lambda).
$$
Note that the notion of “billiard trajectory’’ here might be slightly relaxed: the lower bound only requires closed billiard trajectories to have length $\ge \Omega\text{-}\mathrm{sys}$ when they arise as limits of approximate solutions in the sense of approximation
schemes developed in \cite{AM11,Voc21}.

\paragraph{Question:} \begin{center}
    Is there a sequence of embedded billiard tables
$\Omega_i\subset N$, $i\in\mathbb{N}$, such that
$$
\Omega_i\text{-}\mathrm{sys}\longrightarrow \infty \quad \text{as } i\to\infty\ ?
$$
\end{center}
A positive answer would, to the best of the author’s knowledge, give the first example of a disk cotangent bundle with infinite Hofer--Zehnder capacity.



%% file: Contributions/Umberto.tex


\section*{Bernhard Albach, Umberto Hryniewicz, and Alexandru Oancea}

In~\cite{Albach} it has been conjectured that for a Reeb flow on a closed $3$-manifold, either there are finitely many periodic orbits, or the number of periodic orbits grows at least quadratically. To be precise, consider for each $T>0$ the number $0\leq P(T)\leq \infty$ of periodic orbits with period $\leq T$. It is important to stress that we only count primitive periodic orbits, i.e., we do not count multiply covered orbits. 

\noindent \textbf{Conjecture 1.} For a Reeb flow on a closed connected $3$-manifold, either there are exactly two periodic orbits, or $$ \liminf_{T\to+\infty} \frac{\log P(T)}{\log T} \geq 2 $$ holds.

The conjecture above, which is also cited in~\cite{DanCG}, has been confirmed in~\cite{Albach} for geodesic flows of reversible Finsler metrics on $S^2$. This significantly improves the best known growth rate $$ \liminf_{T\to+\infty} P(T)\frac{\log T}{T} > 0 $$ which was due to Hingston~\cite{Hingston1993}. In fact, in this case there is always quadratic growth since the number of closed geodesics is always infinite. On closed connected orientable reversible Finsler surfaces of genus $\geq 1$, Conjecture~1 is long known to be true for the geodesic flow. Hence, there is always quadratic growth of closed geodesics in any reversible Finsler closed surface -- orientable or not. Geodesic flows on $S^2$ can be lifted to Reeb flows on the standard contact $3$-sphere. It should be noted that the results from~\cite{Albach} cover a significantly larger class of such Reeb flows.

One might wonder if the same holds for a general low-dimensional symplectic dynamical system, which means either a symplectic diffeomorphism on a compact symplectic surface (with or without boundary), or the Hamiltonian flow on a $3$-dimensional compact regular energy level without boundary. In~\cite{Albach} it is also proved, via a simple argument based on the transverse foliations by Le Calvez and Tal~\cite{LeCalvez-Tal}, that the same dichotomy as in Conjecture~1 holds for any symplectic diffeomorphism on~$S^2$.

The confirmation of Conjecture~1 is, however, \textbf{not} the problem proposed here. Instead, we would like to propose the problem of investigating whether the same phenomenon holds in any dimension. 

\noindent \textbf{Problem 1.} Is it true that for a Hamiltonian diffeomorphism on a closed symplectic manifold, either the number of periodic orbits is finite, or periodic orbits grow at least quadratically in the sense that $\liminf_{T\to+\infty} \frac{\log P(T)}{\log T} \geq 2$?

Due to the lack of counterexamples, we suspect the answer might be positive. For instance, as test cases, we propose two particular instances of the above.

\noindent \textbf{Problem 2.} Is it true that $\liminf_{T\to+\infty} \frac{\log P(T)}{\log T} \geq 2$ holds for every Hamiltonian diffeomorphism on a standard symplectic torus? Or in an arbitrary aspherical closed symplectic manifold?

\noindent \textbf{Problem 3.} Is it true that for a Hamiltonian diffeomorphism on $\mathbf{C}P^n$ with infinitely many periodic orbits the inequality $\liminf_{T\to+\infty} \frac{\log P(T)}{\log T} \geq 2$ holds?

In low-dimensions, integrable Hamiltonian systems indicate that quadratic growth should be sharp. To justify this, if one would count invariant tori foliated by periodic orbits or isolated closed geodesics on a closed Riemannian $2$-sphere of revolution, instead of only closed geodesics, then it is simple to present examples where the growth is exactly quadratic. One would conjecture that it is possible to perturb such a system to get existence of systems with exact quadratic growth of periodic orbits. The corresponding statement is also expected to be true for integrable annulus maps, as pointed out by Gerhard Knieper. The construction of such a perturbation is not trivial and can potentially be a very difficult task.

The question that we are asking in higher dimensions is to be understood in the context of the question asked by Hofer and Zehnder in~\cite[p.~263]{HZ}, whether every Hamiltonian map on a compact symplectic manifold $(M,\omega)$ possessing more fixed points than necessarily required by the Arnold conjecture possesses always infinitely many periodic orbits. This question came to be known as the Hofer-Zehnder conjecture. 

Shelukhin proved this conjecture in~\cite{Shelukhin} for closed monotone symplectic manifolds with semi-simple even-dimensional quantum homology ring, a class that includes complex projective spaces. His proof provides the estimate $\liminf_{T\to+\infty} P(T) \frac{\log T}{T}>0$, i.e., the growth rate of the number of periodic points is at least the growth rate of the prime numbers, just like in~\cite{Hingston1993}. 

Ginzburg and G\"urel interpreted more broadly the question of Hofer and Zehnder, in the sense that the existence of a fixed point not strictly required by Floer theory forces the existence of infinitely many periodic points, for example the existence of a hyperbolic fixed point~\cite{Ginzburg-Gurel2014}, or the existence of a non-contractible fixed point~\cite{Gurel2013,Ginzburg-Gurel2016}. Notably, the estimate $\liminf_{T\to+\infty} P(T) \frac{\log T}{T}>0$ is already proved by G\"urel in~\cite{Gurel2013}. One could venture to ask whether quadratic growth holds in such cases as well.

%% file: Contributions/Khesin.tex


\section*{Boris Khesin}

\textbf{1. Find physically interesting examples of infinite-dimensional billiards}

\smallskip

{\bf Motivation:} 
It is a well-known folklore (made into a theorem under certain conditions) that billiard and geodesic flows are closely related by taking appropriate limits, see e.g. \cite{Giannetto, Lange}. There are various physical problems leading to the construction of geodesics on infinite-dimensional spaces. One of the sources is the Euler-Arnold equation describing geodesics on infinite-dimensional groups. In particular, solutions of the Euler equation of incompressible hydrodynamics can be regarded as geodesics on the group of volume-preserving diffeomorphisms with respect to the right-invariant $L^2$-metric on the group. A similar framework exists for the KdV equation, magnetohydrodynamics, Heisenberg chain, etc. \cite{Arnold66, ArnKh}
A natural question is whether there exist interesting infinite-dimensional billiards, possibly related to equations of mathematical physics.

One possible approach would be to consider not one or several particles hitting a wall with a billiard reflection, but a continuum of particles doing that. The difficulty is that then the reflected flow will be passing through the same domain as the flow before reflection, and hence there will be two different flows passing simultaneously through each point in the vicinity of the boundary. However, this is exactly the setting handled by multi-phase flows and by the corresponding groupoid structure for generalized flows, see \cite{Brenier, IzKh}. This way infinite-dimensional billiards in fluids might correspond to a choice of appropriate boundary conditions in the diffeomorphism groupoid for two-phase fluids. It would be very interesting to develop the corresponding theory for multi-phase fluids with billiard reflections from the boundary.

\medskip

\textbf{2. Are there nontrivial examples of integrable pensive billiards?}

\smallskip

{\bf Motivation:} In \cite{DGK}, a new discrete dynamical system generalizing classical billiards was introduced. Given a planar domain with a boundary curve $\gamma$ and a delay function $\ell(\alpha)$, the {\it pensive billiard map} is a composition of the classical billiard map with the translation along the boundary by the distance $\ell(\alpha)$, where this translation may depend on the incidence angle $\alpha$; see Figure \ref{figPensive1}. 

\begin{figure}[h!]\centering
  \includegraphics[width=2.5in]{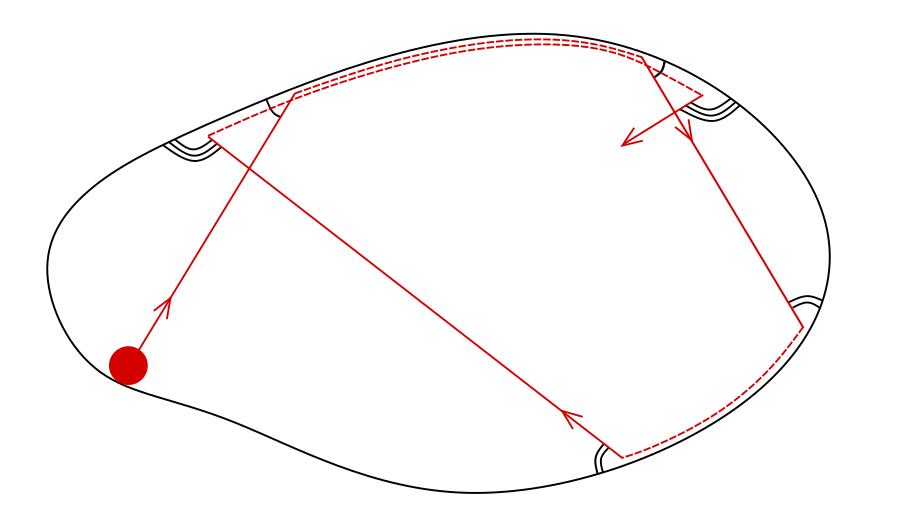} 
  \caption{Pensive Billiard.}
\label{figPensive1}
\end{figure}

One of the systems that admits such a description is a puck (or coin) billiard proposed by M.~Bialy in \cite{Bialy}: geodesics on a surface obtained by gluing two copies of a given domain $D$ to the cylinder $[0,h] \times \partial D$ can be identified with pensive billiard trajectories for an appropriate delay function.
Another example arises in studying 2-dimensional ideal hydrodynamics as a motion of vortex dipoles hitting the boundary of a fluid domain at a certain angle: they split and travel along the boundary in opposite directions until they meet again and merge into a traveling dipole, reflecting from the boundary at the same angle.
 
While the classical billiard, corresponding to the vanishing delay function $\ell(\alpha)\equiv0$, has ellipses as examples of integrable dynamical systems, no interesting integrable cases are known for nontrivial delay. (There are, of course, easy examples  of any delay function $\ell$ on a circle $\gamma$ or the delay $\ell(\alpha)\equiv({\rm perimeter})/2$ on the ellipse.) 
It was recently proved in \cite{Barbieri} that for sufficiently tall puck billiards (i.e., for sufficiently large $h$), there are no caustics, and hence no integrability is possible in that case. 
The question is to find if there exist integrable pensive billiards and present some nontrivial examples by choosing appropriate curves $\gamma=\partial D$ and delays $\ell$ on them.

\medskip

\textbf{3. Is there a fluid of negative density?}

\smallskip

{\bf Motivation:} A mild reformulation of the question above is as follows: {\it Study continuous versions for the interactions of particles of positive and negative mass. Explore the relation with pseudo-Euclidean metrics.} 

To motivate this question, we start with the motion of three point-like beads of masses $m_1, m_2$, and $m_3$
on a frictionless ring, making elastic collisions
with each other. By using the momentum and energy conservation laws and quotienting out the rotational symmetry, the motion of the beads is equivalent to a billiard trajectory in an acute triangle
with angles given by
$$
 \tan \alpha_1=\sqrt{\frac{m_1(m_1+m_2+m_3)}{m_2 m_3}},
 $$
as well as respective cyclic permutations of indices. (Note that if one of the masses, say $m_3$, goes to infinity,
this corresponds to a billiard in a right triangle, $\alpha_3=\pi/2$, which is equivalent to two beads constrained to a segment between two walls.) This equivalence allows one to answer various questions on the existence of periodic orbits, etc.

A natural question: ``Does there exist a mechanical system corresponding to a billiard in an obtuse triangle?'' has a curious answer. This would be a system where one allows beads of {\it negative mass}, as long as they satisfy $(m_1 + m_2 + m_3)m_1 m_2 m_3 > 0$, see \cite{GlashowMittag}. Thus, if the masses have the same sign, then the associated triangle is acute; otherwise, it is obtuse. In this way, the obtuse triangular billiard tables suggest considering the one-dimensional dynamics of particles with positive and negative mass. For instance, in a nice paper \cite{Artigue}, Artigue presents a study of all possible cases of positive and negative masses. One 
introduces a small particle of negative mass, ``graviton'', which via elastic collisions, pulls together other positive masses (and might be the source of gravity in general!), cf.  \cite{Wiki}.

However, there is a different point of view: instead of changing the angles, one can consider the billiard in a fixed triangle on the plane with an appropriately rescaled Euclidean metric. In the case of negative mass(es), the metric becomes pseudo-Euclidean. Billiards in pseudo-Euclidean (or, more generally, in pseudo-Riemannian) 
metrics were described in \cite{KhTab}. The corresponding theory and applications to gravitons are described in \cite{Artigue}. 
The problem is to construct a continuous version of this theory of particles of negative mass. Are there some interesting phenomena related to the corresponding Euler-like equations for the ``fluid of negative density''?

\bigskip

{\it Acknowledgments.} BK thanks the organizers and participants of the workshop on {\it Billiards and Quantitative Symplectic Geometry} at the Heidelberg University for a stimulating atmosphere. The research was partially supported by an NSERC Discovery Grant.

%% file: Contributions/Moreno.tex


\section*{Agustin Moreno}

The set of oriented lines $\mathcal{L}$ in $\mathbb R^{n+1}$ is a symplectic manifold $(\mathcal{L},\omega)$, symplectomorphic to $(T^*S^n,\omega_{std})$. Given a convex billiard table $T\subset \mathbb R^{n+1}$, the induced \emph{billiard domain} is
$$
W_T=\{l \in \mathcal{L}: l \cap T\neq \emptyset\}\subset \mathcal L,
$$
which carries the symplectic form $\omega_T\coloneq\omega\vert_{W_T}$. $(W_T,\omega_T)$ is symplectomorphic (via the identification $\mathcal L \cong T^*S^n$) to a fiber-wise convex domain in $T^*S^n$, and therefore diffeomorphic to $\mathbb D^*S^n$. The associated \emph{billiard map} is
$$
f_T: (W_T,\omega_T)\rightarrow (W_T,\omega_T),
$$
which maps a line intersecting $T$ to its billiard reflection, as determined by $T$. This map preserves the symplectic form $\omega_T$. Moreover, while Hamiltonian (and smooth) in the interior, $f_T$ is only continuous along the boundary $\partial W_T$. We call such a map \emph{$C^0$-Hamiltonian}. The following are open problems:

\begin{enumerate}
    \item Characterize the fiber-wise convex domains in $T^*S^n$ which arise from billiard tables.
    \item Characterize the $C^0$-Hamiltonian maps $f: W \rightarrow W$ of a fiber-wise convex domain $W$ which arise as billiard maps of convex billiard tables.
\end{enumerate}

To the best of the author's knowledge, even the case $n=1$ is open. Moreover, in the latter case, the corresponding billiard maps are classical twist maps. These corresponding billiard maps are also expected to satisfy a version of the twist condition as introduced in \cite{GPB}, or the more relaxed version introduced in \cite{LM25}, although it has not been checked. 

There are connections to celestial mechanics, as it was shown in \cite{MvK} that the dynamics of the (low-energy, near-primary) circular restricted three-body problem, is encoded by a $C^0$-Hamiltonian map of a fiber-wise star-shaped domain in $T^*S^2$. This domain is not necessarily fiber-wise convex, except for a suitable set of parameters known as the \emph{convexity range}.

%% file: Contributions/Oancea.tex


\section*{Jean-Fran\c{c}ois Barraud, Mihai Damian, Vincent Humilière, and Alexandru Oancea}

\textbf{Problem statement} Let $M$ be a closed manifold that is not a $K(\pi,1)$, and endow the cotangent bundle $T^*M$ with its standard symplectic form. We conjecture that, for any closed and connected hypersurface $\Sigma\subset T^*M$ and any Hamiltonian $H:T^*M\to\mathbf{R}$ that admits $\Sigma$ as a regular energy level $\Sigma=H^{-1}(c)$, there exists $\varepsilon>0$ such that almost all energy levels of $H$ in the interval $(c-\varepsilon,c+\varepsilon)$ carry a contractible periodic orbit of $H$.    
 
\textbf{Context and motivation}. 

The study of closed characteristics on hypersurfaces of symplectic manifolds is a fundamental topic in symplectic geometry. When such a closed hypersurface is of contact type, the famous Weinstein conjecture \cite{Weinstein79}, first proved by Viterbo in $\mathbf{R}^{2n}$~\cite{Viterbo-WeinsteinR2n}, claims the existence of at least one closed characteristic.  Without assuming the contact type hypothesis,  given some symplectic manifold $(W,\omega)$ and some proper autonomous smooth Hamiltonian $H:W\to \mathbf{R}$, one can ask which regular levels $H^{-1}(c)$ contain at least one closed characteristic. The question of whether {\it all} the regular levels of any Hamiltonian satisfy this property is known as the Hamiltonian Seifert conjecture. It was disproved already for $W=\mathbf{R}^{2n}$ by Ginzburg~\cite{Ginzburg-Ham-Seifert1} \cite{Ginzburg-Ham-Seifert2} and Herman~\cite{Herman} for $2n\geq6$; G\"urel and  Ginzburg later provided a ${\cal C}^2$ counterexample for $2n=4$ \cite{Ginzburg-Gurel1}, \cite{Ginzburg-Gurel2}. Once the conjecture had been disproved, it is natural to ask whether at least some of the regular levels $H^{-1}(c)$ contain closed characteristics. In the particular case $W=\mathbf{R}^{2n}$ endowed with the standard symplectic form it actually turns out that many of them do. More precisely, a celebrated result of Hofer-Zehnder~\cite{HZ} and Struwe~\cite{Struwe} asserts that almost all (regular) levels of a smooth proper Hamiltonian $H:\mathbf{R}^{2n}\to \mathbf{R}$ carry at least one closed characteristic. This result is known as {\it the almost existence property}.

In a cotangent bundle any closed and connected hypersurface separates $T^*M$ into a bounded component and an unbounded component. Hofer and Viterbo~\cite{HV88} proved that the above property holds for all hypersurfaces that enclose the $0$-section, without contractibility. Albers-Frauenfelder-Oancea~\cite{AFO17} proved the above property under the assumption that the Hurewicz map $\pi_2(M)\to H_2(M)$ does not vanish. Their proof uses the multiplicative structure on symplectic homology twisted by suitable rank 1 local coefficients. Frauenfelder-Pajitnov~\cite{Frauenfelder-Pajitnov} proved this property for rationally inessential manifolds $M$, i.e., manifolds whose fundamental class with rational coefficients vanishes under their classifying map to $B\pi_1$. Their proof involves $S^1$-equivariant homology and localization.  

\textbf{Remarks, partial results} 
In our recent paper~\cite{BDHO-cotangent} we have proved this property for any $M$ that is not a $K(\pi,1)$, but for a somewhat restricted class of hypersurfaces $\Sigma$. Our proof uses a new version of symplectic homology with differential graded local coefficients. 
 
Our hope is that the methods of~\cite{AFO17, Frauenfelder-Pajitnov, BDHO-cotangent}, which are complementary to each other, can be combined in order to prove the above conjecture. 
 
In all cases in which it is known to hold for all closed hypersurfaces, the contractible almost-existence property is a consequence of the finiteness of the so-called \emph{$\pi_1$-sensitive Hofer-Zehnder capacity} for the disc cotangent bundle $D^*Q$. One could also conjecture that this capacity is finite for all manifolds $Q$ that are not $K(\pi,1)$.
 
The assumption of not being $K(\pi,1)$ seems reasonable as a borderline between existence and non-existence of contractible periodic Hamiltonian orbits: the standard class of manifolds that \emph{are} $K(\pi,1)$ is that of manifolds that carry a Riemannian metric of non-positive sectional curvature, and such metrics do not admit contractible closed geodesics. In particular, none of the positive radius sphere bundles for such a Riemannian metric carries any contractible closed characteristic. Thus, in the above conjecture, the contractibility requirement is intimately tied to the non-$K(\pi,1)$ assumption.\\
 
  



  

%% file: Contributions/Paris-Romaskevich.tex


\section*{Olga Paris-Romaskevich}

\subsection*{Tiling billiards}
A tiling billiard is a dynamical system associated with a plane polygonal tiling. Here is the rule: a particle moves in a straight line until it reaches an edge of a polygon; then it crosses into a neighboring polygon with its velocity direction changed according to Snell's law of refraction with refraction coefficient $-1$. Several families of tiling billiards have been studied since $2015$. Here is the full (as far as I know) list of them. For this list, a tiling billiard's name is followed by a description of the underlying tiling, the generic behavior of trajectories, and some bibliography.
\begin{itemize}
    \item [1.] \textbf{Trihexagonal tiling billiard} (a periodic tiling by regular hexagons and triangles) has orbits that are dense in an unbounded periodic open set, as proven in the work of Diana Davis and Pat Hooper \cite{DH18}.
    
    \item [2.] \textbf{Triangle tiling billiards} (a periodic tiling by one triangular tile via central symmetries) have orbits that are typically either periodic or escape linearly to infinity (we call this \textbf{integrable behavior}). There exists a zero-measure set of exceptional triangle tilings (parametrized by the Rauzy gasket\footnote{The Rauzy gasket has now been extensively studied since it appears as the parameter set of exceptional trajectories in various dynamical and topological settings, see, in particular, \cite{AR91, AHS, DDe}. For a definition of the Rauzy gaskets, an attractor of a multidimensional continued fraction algorithm, and an introduction to their dynamical properties, see \cite{AS}.}) that permit trajectories escaping non-linearly\footnote{These trajectories necessarily pass by the circumcenters of all tiles that they cross.}. This has been shown in a series of works \cite{BDFI18, TAF, H_Paro}, see also \cite{highlights_O}. For a quick introduction to the triangle tiling billiards and their properties, I recommend a short animation movie by Ofir David \cite{Ofir}, based on the research results in \cite{TAF}.
    
    \item [3.] \textbf{Cyclic quadrilateral tiling billiards} (a periodic tiling by one cyclic quadrilateral via central symmetries) have integrable behavior, see \cite{H_Paro, TAF}. The set of exceptional tilings (admitting non-linearly escaping trajectories) has been recently classified in \cite{Dynnikovetal} and is parametrized by a fractal set named Novikov's gasket.
    
    \item [4.] \textbf{Double triangle tiling billiards} (a periodic tiling by a pair of triangles $\Delta$ and $\Delta'$ with one common edge $a$, and such that their angles $\alpha$ and $\alpha'$ opposite to this edge are equal) are generalizations of triangle tiling billiards. Partial results on the dynamics have been obtained by Magali Jay in her thesis \cite{Jay_thesis}. Probably, their behavior is integrable, as in the two previous examples.
    
    \item [5.] \textbf{Refractive wind-tree model} (the tiling is defined by rectangles of size $a \times b$ with centers belonging to a lattice $\Lambda$, and the complement of their union (an \guillemotleft{} unbounded polygon \guillemotright{})) has its generic trajectories trapped in bands, see \cite{Jay_thesis}.
    
    \item [6.] \textbf{Generalized tiling billiards} are a generalization of a tiling billiard corresponding to one fixed polygon that does not necessarily tile the plane. A trajectory starting in a copy of $P$ can still be followed by reflecting $P$ by central symmetry with respect to the midpoint of a crossed side, thus producing a local refraction of a trajectory with coefficient $-1$. \textbf{Cyclic pentagonal generalized tiling billiards} show a behavior different from triangle and cyclic quadrilateral tiling billiards. In particular, the orbits may escape in a non-linear way in open sets: they have a principal direction but deviate from it, see \cite{Jay_pentagons} and \cite{Jay_thesis}. The generic behavior is not yet known. 
    
    \item[7.] \textbf{Brick tiling billiards} (a periodic tiling defined by two rows of $1 \times 1$ squares that are shifted relative to one another by some parameter $\theta \in (0,1)$) have all of their orbits periodic when $\theta \in \mathbb{Q}$, as shown in \cite{thesis_park}. The behavior for $\theta \notin \mathbb{Q}$ has not yet been studied.
\end{itemize}

\begin{figure}
    \centering
    \includegraphics[height= 4 cm]{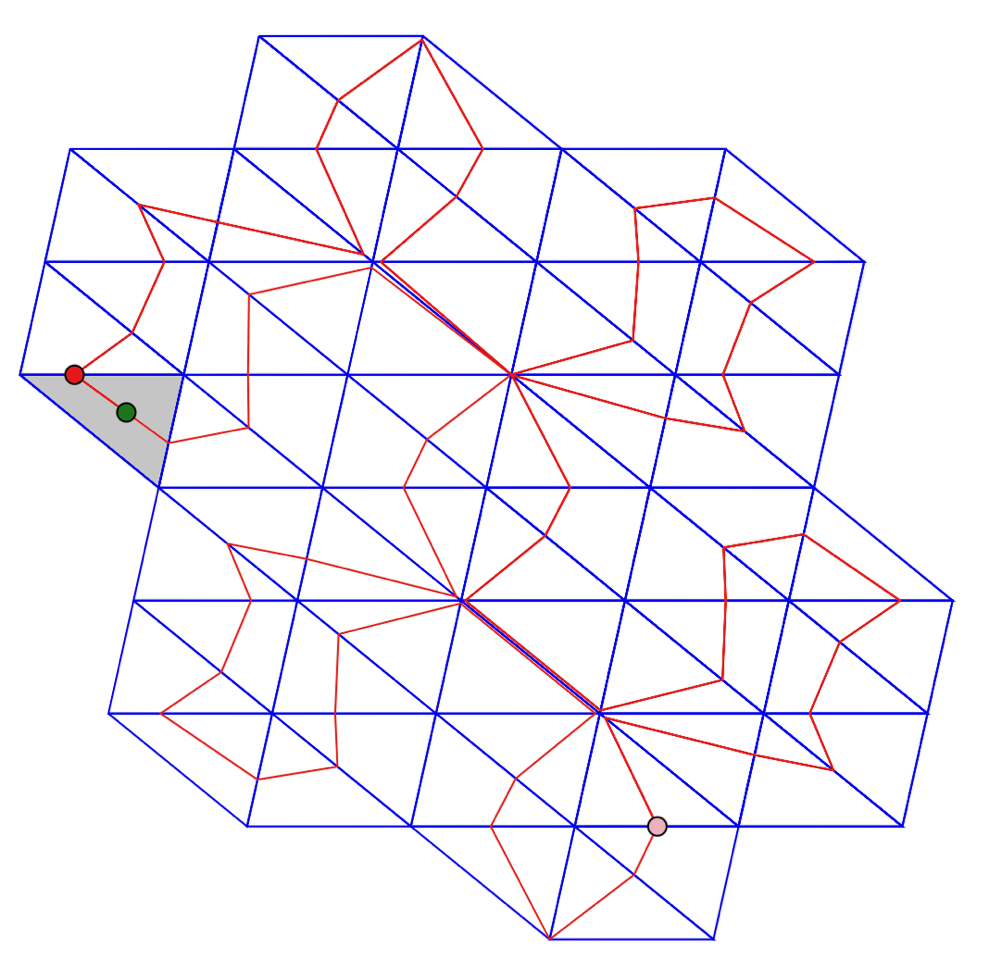}
    \includegraphics[height= 3 cm]{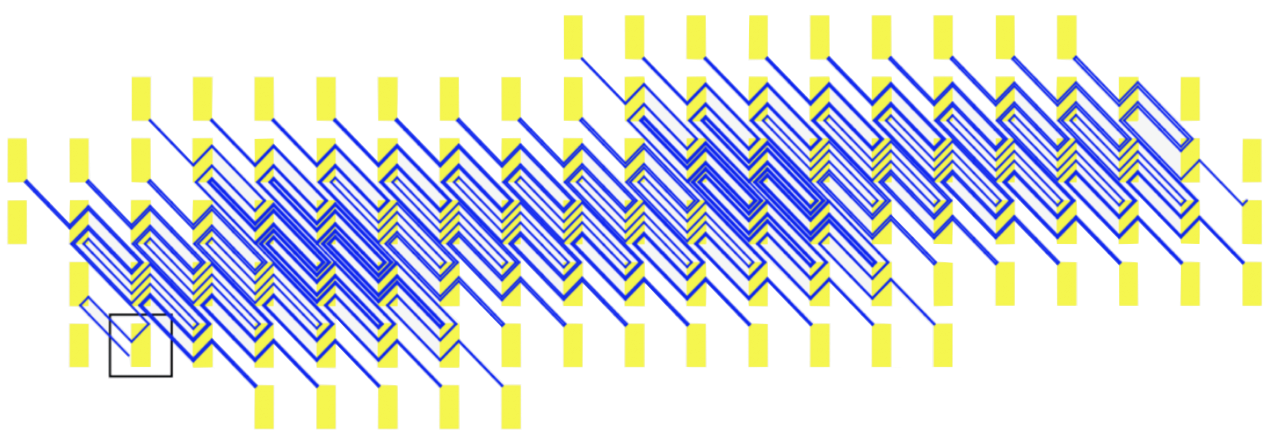}
    \caption{On the left, a periodic trajectory of a triangle tiling billiard. On the right, a trajectory of a refractive wind-tree model confined to a band.}
    \label{fig:examples}
\end{figure}

A tiling of the plane by polygons is \textbf{locally foldable} if it is $2$-colorable and the \textbf{defect of a vertex} (the alternating sum of angles in that vertex) is equal to $0$. This property implies (see \cite{TAF}) strong restrictions on the dynamics of the corresponding billiard: all trajectories are either periodic or escaping, and any trajectory passes through each tile at most once. Moreover, the study of the dynamics of tiling billiards in locally foldable tilings can be related to Novikov's problem in topology. See \cite{Dynnikov_main} for an introduction to Novikov's problem and \cite{PaRoH, Dynnikovetal} for its link to tiling billiards. This is the case for examples $2., 3.$, and $4.$ above.

It is worth mentioning that all of the work on tiling billiards mentioned above was accomplished by reducing their dynamics to that of linear flows on certain families of translation or half-translation surfaces, although the construction method is ad hoc for each family.

\subsection*{Two questions}
In the following, I ask two questions, related in spirit -- one general and one specific.

\begin{question} In the class of tiling billiards in periodic tilings, is there a finite list of generic archetypal behaviors? For instance, the tiling billiards for which generic behavior is understood fall into one of two categories: either their typical trajectories are dense in open sets of the plane, as in cases 1. and 5., or the behavior is integrable (typical trajectories are either periodic or linearly escaping), as in cases 2. and 3. (and conjecturally in case 4. as well). One could ask whether such a dichotomy holds in general.
\end{question}

Locally foldable periodic tiling billiards, up to some technical regularity conditions, are integrable. This is not at all obvious but follows from fine Morse theory developed by Ivan Dynnikov, see in particular Theorem $1$ in \cite{Dynnikov_main}, and \cite{PaRoH}. The first ingredient in the proof comes from \cite{BDFI18} and was further developed in \cite{TAF, PaRoH}. It states that for any locally foldable tiling there exists a \textbf{global folding map} $\mathcal{F}$ defined on a tiled plane such that, on the restriction of two neighboring tiles, it is a folding of one of them onto the other along the common edge (which is an axial symmetry with respect to this edge). 

It seems that the tilings on which one could start testing Question $1$ should be the simplest not-locally foldable tilings - for example, irrational brick tiling billiards or the parallelogram tiling billiards that we discuss in Question $2$.

\begin{question}
Describe the dynamics of (at least, generic) trajectories of \textbf{parallelogram tiling billiards}, defined in a tiling of the plane by parallelograms with sides parallel to the vectors $\Vec{\boldsymbol{a}}:=(a,0)$ and $\Vec{\boldsymbol{b}}:=(b \sin \alpha, b \cos \alpha)$, where $a,b \in \mathbb{R}^+_*$ and $\alpha \in (0, \pi)$. The tiling billiard dynamics depends on two parameters, $\alpha$ and $m=\frac{a}{b}$.
\end{question}

\begin{figure}
    \centering
    \includegraphics[height= 4cm]{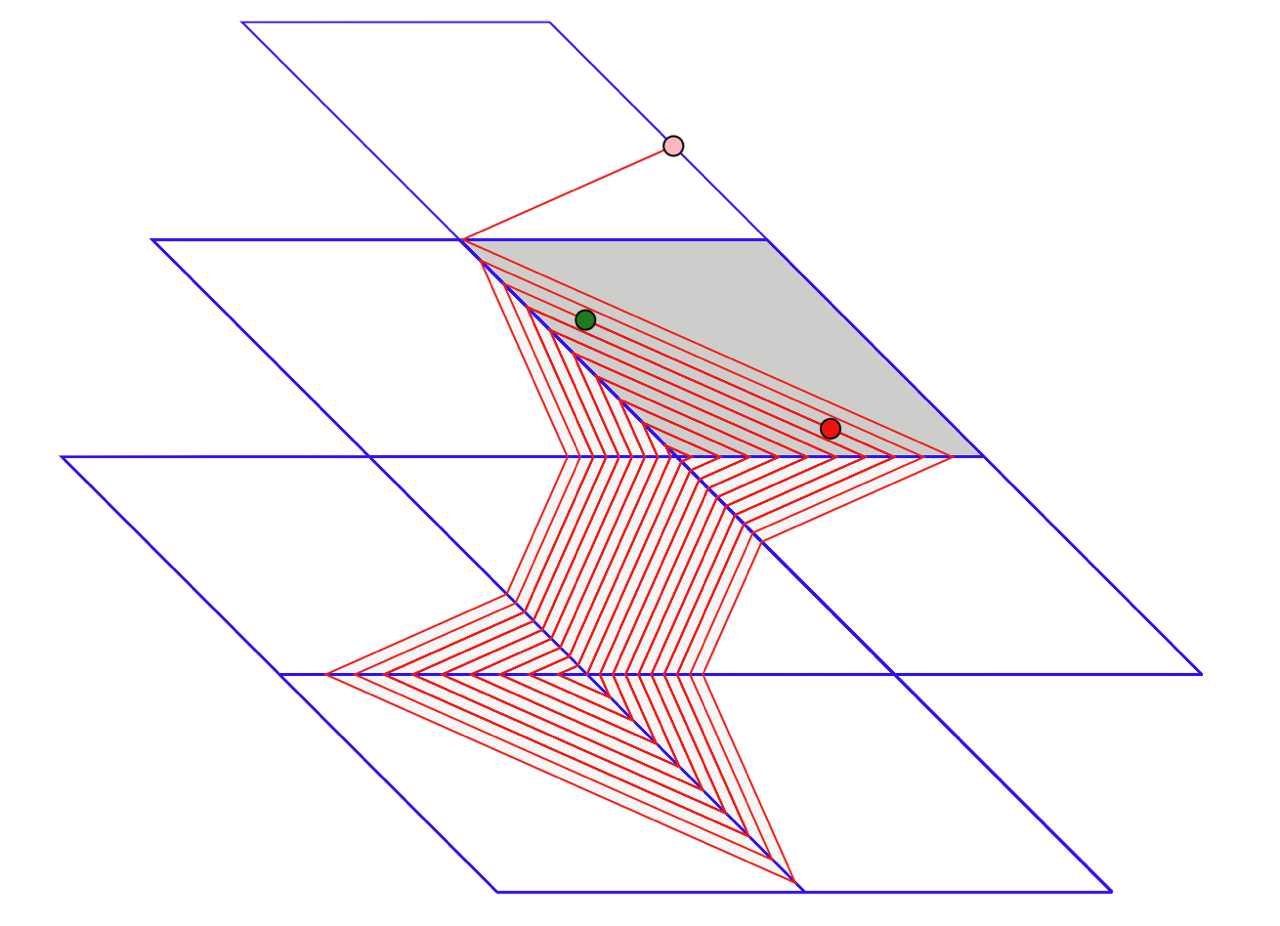}
    \includegraphics[height= 4.5cm]{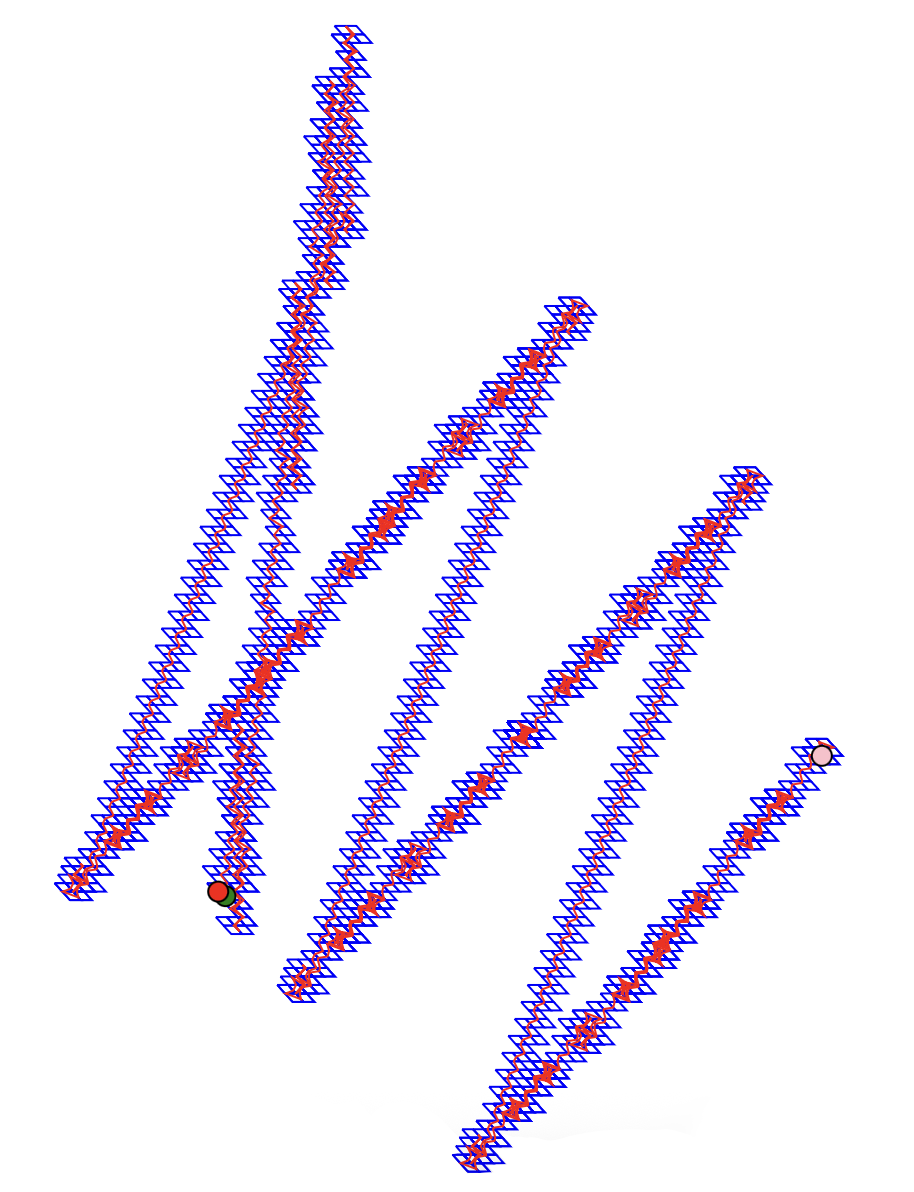}
    \caption{On the left, a trajectory in a parallelogram tiling with parameters $\alpha=\frac{\pi}{4}$ and $a=b=1$ (a trajectory is a leaf of a non-orientable foliation of the plane) after $138$ refractions. On the right, a trajectory with $a=1, b=1.05189074$ and $\alpha=47.8695014$ after $3000$ refractions (a trajectory re-intersects itself).}
    \label{fig:parallelogram}
\end{figure}

The defect of a parallelogram tiling is $4 \alpha - 2 \pi$, so it is never foldable unless the tile is a square. However, the universal cover $\widetilde{X}$ of the plane deprived of all vertices of the tiling can be \guillemotleft{} foldable \guillemotright{} under some conditions on $\alpha$ (obviously, $\alpha = \frac{p}{q} \frac{\pi}{2}$, with $p, q \in \mathbb{N}$ and $(p, q) = 1$, is necessary). More precisely, fix a tile $P_0$. One then defines a \textbf{folding map}
$
    \mathrm{fold} : 
    F(\widetilde{X}) \rightarrow \mathrm{Isom} \; \mathrm{Aff} (\mathbb{R}^2)
$
that associates to each face $f \in  F(\widetilde{X}) $ an isometry of the plane (possibly changing orientation), obtained by folding \guillemotleft{} in an accordion \guillemotright{} along the edges of a path of parallelograms starting in $P_0$, defined by a sequence of crossed edges $e_1, \ldots, e_n$. Denote $P_{(e_n, \ldots, e_1)}$ the last face of this path. Then $\mathrm{fold} (P_{(e_n, \ldots, e_1)}) := s_{e_1} \circ \ldots \circ s_{e_n}$, where $s_e$ is the axial symmetry along the edge $e$. One also defines the map $\mathrm{fold_L} : F(\Tilde{X}) \rightarrow O_2$ by reducing the image of $\mathrm{fold}$ to its orthogonal part.

The goal is then to determine the kernel of $\mathrm{fold}_L$ as well as the image $\mathrm{fold} \left(\mathrm{ker} (\mathrm{fold}_L)\right)$, which consists of translations.
Moreover,

\begin{itemize}
    \item [C1.] under the condition that $ \mathrm{im} \left(\mathrm{fold}\right) $ is finite, the covering $Y:= \widetilde{X} / \mathrm{ker} (\mathrm{fold}_L)$ will be a finite covering (\textbf{Condition 1});
    \item[C2.] and under the condition that $\mathrm{fold} \left(\mathrm{ker} ( \mathrm{fold}_L)\right) =\mathbb{Z} \Vec{u} \oplus \mathbb{Z} \Vec{v} \subset \mathbb{R}^2$, with $\Vec{u}$ and $\Vec{v}$ not collinear (\textbf{Condition 2}),
\end{itemize}    
one obtains a well-defined folding into a torus: $Y \rightarrow \mathbb{R}^2 / \mathbb{Z} \Vec{u} \oplus  \mathbb{Z} \Vec{v}.$ If Conditions C1 and C2 hold, we say that the universal cover $\widetilde{X}$ is \textbf{foldable}.

We omit here rather lengthy calculations of $\mathrm{ker} (\mathrm{fold}_L)$ as well of its image under $\mathrm{fold}$ that we need to be a lattice that we have done, in order to prove the following 

\begin{proposition}\label{Prop_angles_for_cover}[{Théo Marty, Olga P.-R.}]
The universal cover $\widetilde{X}$ of a parallelogram-tiled plane deprived of the vertices of the parallelograms with angles $\alpha$ is well defined if and only if the following two conditions hold:
\begin{itemize}
\item[1.] $\alpha$ is one of the following angles:
$\frac{\pi}{12}, \frac{5 \pi}{12}, \frac{\pi}{6}, \frac{\pi}{3}, \frac{\pi}{8}, \frac{3\pi}{8}, \frac{\pi}{4}$;
\item[2.] the set $\mathbb{Q} \Vec{\boldsymbol{a}} \oplus \mathbb{Q} \Vec{\boldsymbol{b}}$ is preserved by the rotation of angle $4 \alpha$.
\end{itemize}
\end{proposition}

Note that $\alpha=\frac{\pi}{4}$ is the simplest case: the second condition holds for all $a$ and $b$. In this case, the tiling billiard trajectories form non-orientable parallel foliations on the tiled plane.

I think that the foldability of the cover might make the study of such tiling billiards easier and reduce it to the study of dynamics on translation surfaces. One could then imagine perturbing these angles slightly to see how the dynamics would change. One could even (largely) speculate that there might be some KAM-theory effects, with the parallelograms appearing in Proposition \ref{Prop_angles_for_cover} playing the role of analogues of the KAM Diophantine tori.

\subsection*{Acknowledgments}
Proposition \ref{Prop_angles_for_cover} was proven in collaboration with Théo Marty, and the description of Question $2$ is partially based on his notes. For those interested in tiling billiards, the program NegSnell \cite{program}, developed by Alexander StLaurent and Patrick Hooper, simulates multiple tiling billiard trajectories and can be a good place to start to gain some intuition, as well as to illustrate this contribution.

%% file: Contributions/Sorrentino.tex


\section*{Alfonso Sorrentino}

The mathematical framework of symplectic geometry provides a powerful language for describing conservative mechanical systems. The real world, however, is rarely frictionless. To model dissipative phenomena, the concept of \emph{conformally symplectic dynamics} offers an elegant generalization.

On a symplectic manifold, a flow or map is called {\it conformally symplectic} if it transforms the symplectic form by a constant multiplicative factor. This property suitably captures systems with friction or dissipation that is globally proportional to the momentum by a constant factor, and many interesting physical examples can be described and investigated within this framework. We refer to some recent works in this direction~\cite{AllaisArnaud, ArnaudFejoz, CCD, CCD2, MS} and references therein.

One can naturally consider a more general geometric setting. A \emph{locally conformal symplectic structure} on a manifold is a genuine generalization of a symplectic structure. Locally (on each chart) a conformal symplectic manifold is equivalent to a symplectic manifold, but the local symplectic structure is only well-defined up to scaling by a constant. The monodromy of this local structure around curves can induce non-trivial rescalings. This notion was first introduced by Vaisman~\cite{Vaisman, Vaisman2} and also appeared in the work of Guedira and Lichnerowicz~\cite{GL}. It was later studied by Banyaga~\cite{Banyaga}, and has attracted growing interest in recent years (see, for example,~\cite{CM}).

This more general setting could potentially be used to model a wider class of non-conservative systems with non-uniform or spatially dependent dissipation. However, the list of compelling real-world examples of systems that genuinely fit the framework remains surprisingly empty.

\noindent
\textbf{Problem:} \textit{Provide examples of real-world systems (from physics, mechanics, biology, economics, etc\ldots) that can be modeled as locally conformal symplectic flows or maps.}

In the following, we provide some material that could be useful to address this question. We begin by recalling the formal setting, starting from the well-understood {\it global} conformally symplectic case, namely when the ambient manifold is a symplectic manifold. We then introduce the {\it local} conformally symplectic case and provide some remarks and computations that may be useful for identifying and constructing the sought-after examples.

\subsection{Conformal symplectic dynamics on a symplectic manifold}\label{sec1}

\subsubsection{Setting}
Let $M$ be a finite-dimensional compact and connected smooth manifold, equipped with a smooth Riemannian metric $g$.
We denote by $\omega=-\d\lambda$ the canonical symplectic form on $T^*M$, where $\lambda$ is the Liouville (or tautological) form. Choosing local coordinates $(x_1,\ldots, x_n, p_1,\ldots , p_n)$ on $T^*M$, one has that $\omega= \d x\wedge \d p := \sum_{i=1}^n \d x_i\wedge \d p_i$ and $\lambda=p\d x := \sum_{i=1}^n p_i\d x_i$.\\

\begin{definition}
A smooth vector field $X$ on $T^*M$ is said to be {\it conformally symplectic} (CS) if there exists $c \in {\mathbb R}\setminus\{0\}$ such that $${\mathcal L}_X \omega = c \,\omega$$ where ${\mathcal L}_X$ denotes the Lie derivative in the direction of $X$. Clearly, the symplectic case corresponds to the limit case $c =0$. \\
\end{definition}

\begin{remark}
{\bf (i)} Condition ${\mathcal L}_X \omega = c \,\omega$ can be equivalently rewritten (using Cartan's formula) as:
$$
\d (\iota_X \omega) = c \,\omega.
$$
{\bf (ii)} Observe that if $X$ is conformally symplectic, also $-X$ is conformally symplectic and ${\mathcal L}_{-X} \omega = -c \,\omega$.
Hence, up to a time-inversion of the flow, one can always choose the sign of $c$. In particular, the case $c<0$ corresponds to the {\it dissipative} case.\\
Hereafter, we shall consider the case in which
\begin{equation}\label{hypconfsymp}
{\mathcal L}_{X} \omega = -c \,\omega \qquad \mbox{for some}\; c>0.
\end{equation}
\end{remark}

\medskip

\subsubsection{Properties of the flow}
Let us start by studying the properties of conformally symplectic vector fields and by deriving the differential equations that govern the motion induced by $X$. Using Cartan's formula and the closedness of $\omega$, one obtains, denoting $\iota_X$ the inner product, or contraction, with $X$,
$$
{\mathcal L}_X \omega = \d(\iota_X \omega) + \iota_X(\d\omega) = \d(\iota_X \omega).
$$
Hence, the conformally symplectic condition and the exactness of $\omega=-\d\lambda$ imply
$$\d(\iota_X\omega - c \lambda) =0$$
that is, the 1-form $\iota_X\omega - c \lambda$ is closed. We define the {\it cohomology class} of $X$ to be the cohomology class of this 1-form; it will be denoted by $[X] \in H^1(M;{\mathbb R})$. Observe that here (as well as in the following) we are tacitly identifying the de Rham cohomology groups $H^1(M;{\mathbb R})$ and $H^1(T^*M;{\mathbb R})$ by means of the isomorphisms induced by the projection map $\pi:T^*M\longrightarrow M$, which is a homotopy equivalence, and by its homotopy inverse $\iota: M\longrightarrow T^*M$ given by the inclusion of the zero section.\\

\noindent We say that $X$ is {\it exact conformally symplectic} if $[X]=0$. In this case, there exists a smooth function $H: T^*M \longrightarrow {\mathbb R}$ such that
\begin{equation}\label{eq1}
\iota_X\omega - c \lambda = \d H.
\end{equation}
We call this function an {\it Hamiltonian} associated to $X$. Observe, in fact, that in local coordinates $(x_1,\ldots, x_n,p_1,\ldots, p_n)$, relation (\ref{eq1}) becomes:
$$
-\dot{p} \d x + \dot{x} \d p - c p \d x = \frac{\partial H}{\partial x}(x,p) \d x + \frac{\partial H}{\partial p}(x,p) \d p,
$$
and therefore the vector field is given (in local coordinates) by:
\begin{equation}\label{eqmotion}
\left\{
\begin{array}{l}
\dot{x} = \frac{\partial H}{\partial p}(x,p)\\
\dot{p} = - \frac{\partial H}{\partial x}(x,p) - c p.
\end{array}
\right.
\end{equation}

These equations represent a mechanical system that is subject to a friction, which is constantly proportional to its momentum/velocity ({\it i.e.}, the term $-c p$).\\

\begin{remark}
The non-exact case ({\it i.e.}, $[X]= \sigma \neq 0$) can be transformed into an exact one, modulo a suitable symplectic change of coordinates (which is of course non-exact); see \cite[Remark 2]{MS}.
\end{remark}

\bigskip

\subsection{Locally conformal symplectic systems}

\noindent Our wish would be to consider systems in which the dissipation factor is not necessarily constant. Of course, observe that
if we consider a vector field $X$ on $T^*M$ such that
\begin{equation}
{\mathcal L}_X\omega = f \,\omega
\end{equation}
where $f: T^*M\longrightarrow {\mathbb R}$ is a smooth function, then the evolution of the symplectic form by the flow would not be symplectic anymore (it stops being closed, unless $f$ is constant!).
In the constant case, in fact, the symplectic form evolves in time like $e^{-\lambda t} \omega$, which remains closed and non-degenerate for every $t\in {\mathbb R}$. This fails when we replace $\lambda$ by a function.

Hence, one needs to consider a different structure that is preserved by this kind of flow.\\

\subsubsection{Setting}\label{sec2}

Let $N$ be a finite-dimensional compact and connected smooth manifold, equipped with a smooth Riemannian metric $g$.
Let $\Omega^k(N)$ be the set of differential $k$-forms on $N$ and let $\d$ denote the associated exterior differential.\\

\begin{definition}
Let $\eta$ be a closed 1-form on $N$. We define a new exterior differential
\begin{align}
  \d_\eta :\: \Omega^k(N)&\rightarrow\Omega^{k+1}(N) \nonumber \\
           \alpha &\mapsto \d\alpha -\eta\wedge\alpha.
\end{align}
\end{definition}

\noindent Let us check that it is indeed an exterior derivative (sometimes called {\it Lee differential} or {\it Lichnerowicz differential}).\\

\begin{lemma}\label{lemma1}
$$ \d_\eta \circ \d_\eta =0 \qquad \Longleftrightarrow \qquad \d\eta=0.$$
\end{lemma} 

\begin{proof}
Let $\alpha \in \Omega^k(N)$. Then:
\begin{eqnarray*}
\d^2_\eta \alpha &=& \d_\eta \left(
\d\alpha - \eta \wedge \alpha
\right) = \d \left(
\d\alpha - \eta \wedge \alpha
\right) - \eta \wedge \left(
d\alpha - \eta \wedge \alpha
\right) 
\\
&=& \d^2 \alpha - \d (\eta \wedge \alpha) - \eta \wedge \d\alpha + \eta \wedge \left( \eta \wedge \alpha \right)\\
&=& \eta \wedge \d \alpha - \eta \wedge \d\alpha = 0,
\end{eqnarray*}
where in the second-last equality we used that $\d^2=0$ and $\d\eta=0$.
\end{proof}

\bigskip

A natural question is how this exterior differential depends on the choice of $\eta$ in its cohomology class. There is indeed a natural {\it gauge invariance}.\\

\begin{lemma}[Gauge invariance]\label{lemma2}
  Let $\alpha\in\Omega^k(N)$ and $f\in C^\infty(N)$. Then:
$$\d_{\eta+\d f}\alpha= e^f \d_\eta(e^{-f}\alpha).$$
In particular:
$$ \d_{\eta+df}\alpha=0 \Longleftrightarrow \d_\eta(e^{-f}\alpha)=0. \\$$
\end{lemma}

\bigskip
\begin{proof}
\begin{eqnarray*}
\d_{\eta} \left(e^{-f} \alpha \right) &=&
\d\left(e^{-f} \alpha \right) - \eta \wedge \left(e^{-f} \alpha \right) =
e^{-f} \left(-\d f \wedge \alpha + \d\alpha - \eta \wedge \alpha \right)\\
&=&
e^{-f} \left(\d\alpha - (\eta+\d f) \wedge \alpha \right)\\
&=& e^{-f} \d_{\eta+\d f}\alpha.
\end{eqnarray*}
\end{proof}

\noindent The following formula will be useful later.\\

\begin{lemma}[Modified Cartan's formula] \label{lemma3}
  Let $X$ be a smooth vector field on $N$ and $\alpha \in \Omega^k(N)$. Then:
  \[
{\mathcal L}_X\alpha = \d_\eta (\iota_X\alpha) + \iota_X(\d_\eta\alpha) + (\iota_X\eta)\alpha,
  \]
 where $\iota_X$ denotes the interior product ({\it i.e.}, the contraction of the differential form by the vector field).
\end{lemma}

\medskip

\begin{proof} Using Cartan's formula and the definition of $\d_\eta$, we obtain:
\begin{eqnarray*}
{\mathcal L}_X \alpha &=& \d \left(\iota_X \alpha\right ) + \iota_X \d\alpha \\
&=& 
\d_\eta \left(\iota_X \alpha\right ) + \eta\wedge \iota_X \alpha + \iota_X \left( \d_\eta \alpha + \eta \wedge \alpha \right) \\
&=& \d_\eta \left(\iota_X \alpha\right ) + 
\iota_X \left( \d_\eta \alpha \right)
+ (\iota_X\eta)\alpha.
 \end{eqnarray*}
\end{proof}

\bigskip

\begin{definition}[{\bf Locally Conformal symplectic manifolds}]
A {\it conformally symplectic manifold} is a triple $(N, \eta, \omega)$, where:
\begin{itemize}
\item $\eta \in \Omega^1(N)$ and it is $\d$-closed ({\it i.e.}, $\d\eta=0$),
\item $\omega \in \Omega^2(N)$ is non-degenerate and $\d_\eta$-closed ({\it i.e.}, $\d_\eta \omega=0$).\\
 \end{itemize}
 \end{definition}
 
 \medskip

\begin{remark}
{\bf (i)} Observe that the non-degeneracy of $\omega$ implies that the dimension of $N$ must be even.\\
{\bf (ii)} If $\eta=0$ this reduces to the definition of symplectic manifold. Whereas if $\eta= \d f$, this means that $e^{-f}\omega$ is a symplectic form on $N$ (see Lemma \ref{lemma2}).\\
{\bf (iii)} One could consider this conformally symplectic structure, modulo the equivalence relation $\eta \sim \eta + \d f$. \\
{\bf (iv)} 
 It has been proved by Lee \cite{Lee} that $(N,\eta,\omega)$ being locally conformally symplectic is equivalent to the existence of an atlas of charts $N=\bigcup U_i,\: \phi_i:U_i\rightarrow {\mathbb R}^{2n}$ such that the transition maps $\psi_{ij}=\phi_i\circ\phi_j^{-1}$ satisfy $\psi_{ij}^*(\d x\wedge \d p) = c_{ij}(\d x\wedge \d p)$ where $c_{ij}$ are positive local constants, or equivalently $\d(e^{\phi_i}\omega)=0$. Moreover, if $\eta =\d f$ is exact, then $\d(e^f\omega)=0$ and $(N,\eta=\d f,\omega)$ is said \emph{globally conformally symplectic}.\\
 
Geometrically, this can also be interpreted as a line bundle on $N$ and $\d_\eta$ is a flat connection on it (flatness follows from the condition $\d\eta=0$).\\
 \end{remark}
 
 \bigskip

 \begin{definition}[{\bf Conformally symplectic vector fields}]
Let $(N,\eta,\omega)$ be a conformally symplectic manifold and let $X$ be a vector field on $N$.
 We say that $X$ is {\it conformally symplectic} if there exists a smooth function $f: N\rightarrow{\mathbb R}$ such that
  \begin{equation} \label{eqcsgen}
\left\{  \begin{array}{l}
    {\mathcal L}_X\omega = f \omega \\
    {\mathcal L}_X\eta = \d f.
    \end{array}\right.\\
  \end{equation}
  \end{definition}

\medskip

\begin{remark}
{\bf (i)} Let us intuitively explain why we have two conditions. The first equation describes how the form $\omega$ is changed by the flow, exactly as in the case described in section \ref{sec1}. In particular, if we consider the push-forward by flow $\Phi_X^t$, we have: 
$
\omega \longmapsto e^{f(x)t} \omega. 
$
It follows from Lemma \ref{lemma2} that if we want the evolution of the symplectic form to remain ``closed'', we need to change the closed form $\eta$ accordingly, namely
 $\eta \longmapsto \eta+ t \d f$. Therefore, we get the second condition $ {\mathcal L}_X\eta = \d f$. Observe that these conditions ensure that we remain in the same equivalence class of conformally symplectic manifolds. \\
{\bf (ii)} if $\eta=0$ ({\it i.e.}, $\omega$ is already a symplectic form), it follows that $\d f=0$, hence $f$ must be constant. This explains why in the symplectic setting, one can only consider conformally symplectic vector fields with a constant conformal factor.\\
\end{remark}

\medskip

One can rewrite equations \eqref{eqcsgen} in an alternative form, which resembles the constant factor case, modulo changing $\d$ with $\d_\eta$.\\

\begin{lemma}
Let $X$ be a vector field on $N$. $X$ is conformally symplectic as in \eqref{eqcsgen} if and only if
  \begin{equation} \label{condcs}
  \d_\eta \iota_X\omega = c\,\omega \qquad {\rm where}\qquad c:= f- \iota_X\eta \in{\mathbb R}.
  \end{equation}
\end{lemma}

\medskip

\begin{proof}
  Using the modified Cartan's formula (Lemma \ref{lemma3}) and the fact that $\d\eta=0$, we obtain
  \begin{eqnarray*}
    {\mathcal L}_X\eta = \d f &\Longleftrightarrow& \d (\iota_X\eta) = \d f \\
    &\Longleftrightarrow& \d(f-\iota_X\eta) =0 \\
    &\Longleftrightarrow& f- \iota_X\eta =:c \in{\mathbb R}.
  \end{eqnarray*}
  Substituting in the first equation in \eqref{eqcsgen}:
  \begin{eqnarray*}
    {\mathcal L}_X\omega = f\, \omega &\Longleftrightarrow&
     \d_\eta (\iota_X\omega) + (\iota_X\eta)\omega = f\omega \\
    &\Longleftrightarrow& \d_\eta (\iota_X\omega) = (f- \iota_X\eta)\omega\\
    &\Longleftrightarrow& \d_\eta (\iota_X\omega) = c\,\omega.
  \end{eqnarray*}
\end{proof}

\bigskip

\subsubsection{Conformally Symplectic Cotangent Bundles} 

Let us now consider an interesting example of conformally symplectic manifolds. As in section \ref{sec1}, we consider $N=T^*M$, where
$M$ is a finite-dimensional compact and connected smooth Riemannian manifold. Let $\lambda$ be the Liouville (or tautological) form and choose
$\eta \in \Omega^1(T^*M)$ such that $\d \eta =0$. \\

\noindent We consider the {\it exact} conformally symplectic manifold given by $(T^*M, \eta, \omega_\eta)$, where $\omega_\eta:= -\d_\eta \lambda$. Clearly, $\omega_\eta$ is $\d_\eta$-closed (see Lemma \ref{lemma1}).\\

\noindent Let $X$ be a conformally symplectic vector field on $(T^*M, \eta, \omega_\eta)$. We would like to derive the equations that satisfy its flow, as already done in \eqref{eqmotion}.

In this case, condition (\ref{condcs}) reads
 \begin{equation}
 \d_\eta( \iota_X\omega + c\lambda) =0
 \end{equation}
 that is, the 1-form $\iota_X\omega -c\lambda$ is $\d_\eta$-closed.
 Let us start by considering the case in which this form is $\d_\eta$-exact; hence, there exists a smooth
function $H:T^*M\rightarrow{\mathbb R}$ such that
  \begin{align}
    \label{condhcs}
    \iota_X\omega + c\lambda =\d_\eta H. 
  \end{align}
The function $H$ will be called a Hamiltonian.\\
\medskip

We write down condition (\ref{condhcs}) in local coordinates $(x,p)\in T^*M$, so that $\lambda =p\, \d x$. Hence we also denote $X=(\dot{x},\dot{p})$ and $\eta = \bar{\eta}_x \cdot \d x + \bar{\eta}_p \cdot \d p$ where $\bar{\eta}_x(x,p) = (\bar{\eta}_{x,1}(x,p),\dots,\bar{\eta}_{x,n}(x,p))$, $\d x =(dx_1,\dots, dx_n)$ and analogously for $\bar{\eta}_p$ and $\d p$.

We have,
\begin{eqnarray}\label{oneside}
 \iota_X\omega + c\lambda & = & \iota_X(-\d_\eta l) + c\lambda = \iota_X(-\d\lambda + \eta\wedge\lambda ) + c\lambda \nonumber\\
 & =& -\dot{p} \cdot \d x +\dot{x}\cdot \d p + (\iota_X\eta) \lambda - (\iota_X\lambda) \eta + c\lambda \nonumber\\
 & =& -\dot{p} \cdot \d x +\dot{x}\cdot \d p + (f-c) \lambda - (p\cdot\dot{x}) \eta + c\lambda \nonumber\\
 & =& -\dot{p} \cdot \d x +\dot{x}\cdot \d p + f\, p \cdot \d x - (p\cdot\dot{x}) (\bar{\eta}_x \cdot \d x + \bar{\eta}_p \cdot \d p),
\end{eqnarray}
where we used that $\iota_X\eta = f-c$, as it follows from \eqref{condcs}.\\

\noindent On the other side, we have
\begin{equation}\label{theotherside}
\d_\eta H =\d H - H\eta = \partial_x H \cdot \d x + \partial_p H \cdot \d p - H(\bar{\eta}_x \cdot \d x + \bar{\eta}_p \cdot \d p). 
\end{equation}

Putting together \eqref{oneside} and \eqref{theotherside}, we get the following {equations of motion}:
\begin{equation}
  \label{complete_eq}
  \left\{
  \begin{array}{l}
    \dot{x} = \partial_p H + (p\cdot \dot{x} -H) \bar{\eta}_p
    \\
     \dot{p} = -\partial_x H + f \,p - (p\cdot \dot{x} -H) \bar{\eta}_x.
  \end{array}
  \right.
\end{equation}

\bigskip
\bigskip

\noindent {\bf Question:} {\it Is there a physical example that fits into the framework of equations \eqref{complete_eq}? What do equations \eqref{complete_eq} mean from a physical/mechanical point of view? }\\

\noindent Let us assume for simplicity that $\bar{\eta}_p=0$, namely $\eta$ is the pullback via the projection $\pi:T^*M\longrightarrow M$ of a closed form on $M$: $\eta = \bar{\eta}_x \cdot \d x.$ Then, equations (\ref{complete_eq}) become:

\begin{equation}
  \label{ham_eq}
  \left\{
  \begin{array}{l}
    \dot{x} = \partial_p H 
    \\
     \dot{p} = -\partial_x H + f \,p - (p\cdot \dot{x} -H) \bar{\eta}_x.
  \end{array}
  \right.
\end{equation}

\noindent In particular, if $H$ is a Tonelli Hamiltonian, then the term $(p\cdot \dot{x} -H)$ is exactly the Lagrangian associated to $H$; in fact, it follows from the first equation that
$\dot{x} = \partial_p H$. Hence, we could rewrite these equations as:

\begin{equation}
  \label{ham_eq2}
  \left\{
  \begin{array}{l}
    \dot{x} = \partial_p H 
    \\
     \dot{p} = -\partial_x H + f p - L(x,\dot{x}) \bar{\eta}_x.
  \end{array}
  \right.\\
\end{equation}

\medskip
\medskip

\noindent {\bf Question:} {\it The term $f p$ is what we were expecting (some dissipation/repulsion related to the momentum). However: what does the term $ - L(x,\dot{x}) \bar{\eta}_x$ represent? Does it have any interpretation?}

%% file: Contributions/Tabachnikov.tex


\section*{Serge Tabachnikov}

\paragraph{Projective billiards.} Projective billiards were defined in \cite{Ta97}. A projective billiard table is a convex domain with smooth boundary equipped with a transverse field of directions; the phase space is the same as that of the conventional billiard, and the reflection law is such that the tangential component of the incoming velocity vector is preserved, whereas the transverse component is reversed. 

Problem: {\it when does a projective billiard possess an invariant symplectic form?} And, related, {\it when does it have a variational description?}

One such situation is described in \cite{Ta97-1}: the billiard table is bounded by a sphere, and the transverse field satisfies a certain {\it exactness} condition. The invariant symplectic form on oriented lines that intersect the sphere is induced by the hyperbolic metric in its interior (the projective, or Beltrami-Cayley-Klein model).
Another example, in dimension 2, is when the billiard curve is an arbitrary oval, but the lines of the transverse field are all concurrent; the invariant area form comes from the standard area form in the plane via polar duality (see \cite{Ta97}).

In recent papers \cite{Glu,GM}, a somewhat related question was studied: when is a Minkowski billiard projective? (In which case it has a variational description and an invariant symplectic form). The answer is that such a billiard is the conventional one in some Euclidean metric. 

\paragraph{Inner and outer periodic orbits on immersed submanifolds.} Let $M$ be a closed immersed submanifold of Euclidean space. A generalized $n$-periodic billiard trajectory is an $n$-gon inscribed in $M$ (that is, the vertices lie on $M$) that has an extremal perimeter. 

Problem: {\it estimate from below the number of $n$-periodic trajectories in terms of the topology of $M$ and the immersion}. 

The case of $n=2$ was studied in \cite{Pu}, where the answer is given in terms of the topology of $M$ (but not of the immersion). For example, for a 2-dimensional torus, the number of such {\it diameters} is not less than 10, and this estimate is sharp. For higher periods, see \cite{Du1,Du2}.

A similar question can be asked about periodic trajectories of the outer symplectic billiard about $M$; see \cite{ACT} for their recent study.

\paragraph{Complexity of polygonal symplectic billiards.} Symplectic billiards were introduced in \cite{AT}; the polygonal ones were studied in \cite{ABSST,ALW}. 

Label the sides of a polygon by some symbols and assign to a billiard trajectory its symbolic code, the corresponding bi-infinite word in these symbols. The set of all trajectories yields a collection of such words. The complexity function $p(n)$ is the number of different subwords of length $n$ in this collection. 

For the conventional polygonal billiards, it is known that the complexity function is sub-exponential (the literature on this subject is huge, starting with \cite{Ka}), and it has polynomial, actually cubic, growth for rational billiards (whose angles are rational multiples of $\pi$). For polygonal outer billiards, the growth of the complexity function is polynomial, see \cite{GT}.

Problem: {\it study the complexity of polygonal symplectic billiards}. 

Note that there are examples of such billiards with all trajectories being periodic; in such cases, the complexity function has a zero growth rate.

\paragraph{Chaotic symplectic billiards?} Problem: {\it do chaotic symplectic billiards in convex $C^1$-smooth domains exist?}

A natural candidate is the stadium; there is numerical evidence that this might be the case.

\paragraph{Planar characteristic curves.} Let $M$ be (a germ of) a smooth hypersurface in the standard symplectic space. 

Problem: {\it how restrictive is the condition that the characteristic curves on $M$ are planar?}

If $M$ is closed and convex, the situation is rigid: $M$ is an affine symplectic image of a round ball; see \cite{KS}.

\paragraph{Density of reflected wave fronts.} Consider a domain with smooth boundary in Euclidean space and let $P$ be a generic point in its interior. Let $P$ be a source of light, and consider the time $t$-front, that is, the locus of ``photons" emanating from $P$ in all directions, undergoing optical reflection off the boundary, and traveling for time $t$. 

Problem: {\it does this time $t$-front become dense in the limit $t \to \infty$?} 

This problem, and its analog for smooth Riemannian manifolds, is discussed in \cite{KK}. Even the case of a round circle is open. The length growth of wave fronts is studied in \cite{dV,dV2}.

\paragraph{Ivrii's conjecture for outer symplectic billiards.} The Ivrii conjecture for conventional billiards asserts that the set of periodic orbits has zero measure (with respect to the measure associated with the invariant symplectic structure of the phase space). A similar conjecture applies to the outer symplectic billiards.

For conventional billiards, this conjecture is proved for periods $n=2,3$ in all dimensions, and for $n=4$ in dimension 2.
For planar outer billiards, there are proofs for $2n+1$-periodic orbits with rotation number $n$ and for $2n$-periodic orbits with rotation number $n-1$.  See \cite{Sh} and the references therein about these results.

Problem: {\it prove Ivrii's conjecture for 3-periodic multidimensional outer symplectic billiards with respect to a strictly convex closed hypersurface.}

\paragraph{Approximating a convex curve by polygons with respect to the symmetric difference.} A popular topic in convex geometry is the approximation of smooth convex closed curves by polygons, see \cite{Gr}. The quality of approximation by $n$-gons may be measured by the difference of perimeters or areas, and the polygons may be inscribed or circumscribed. This leads to four functions of $n$, and these four functions correspond to the minimal actions of four billiards: the conventional, Birkhoff, billiards, the outer billiards, the symplectic billiards, and the outer length billiards (see \cite{AT24,BBF} about these lesser studied billiards). 

Consider the approximation of an oval by $n$-gons where the quality of approximation is measured by the area of the symmetric difference. 

Problem: {\it is the resulting function of $n$ related to some billiard-like dynamical system?}

\paragraph{Optical transformations of the space of oriented lines.} The billiard (optical) reflection is a (local) symplectic transformation of the space of oriented lines (rays of light); this space carries a symplectic structure induced by the canonical symplectic structure of the cotangent bundle of the ambient  $n$-dimensional space via symplectic reduction.

The space of oriented lines has dimension $2n-2$, and its symplectic transformation depends on a smooth function of $2n-2$ variables (for example, an infinitesimal symplectic transformation, i.e., a Hamiltonian vector field, is defined by its Hamiltonian function).

On the other hand, a smooth mirror is a hypersurface, locally the graph of a function of $n-1$ variables. Such functions constitute a minuscule part of the space of smooth functions of $2n-2$ variables. Thus, one expects the {\it optical} set of (local) symplectic transformations of the space of rays that are realized by a finite collection of mirrors to be a very small part of the set of all (local) symplectic transformations of the space of rays.

Problem: {\it how to characterize this optical symplectic transformation within the space of all symplectic transformations of the space of oriented lines?}

See \cite{PTT,Pe,Ta20,Gl1} for some results in this direction. 